\documentclass[reqno]{amsart}

\usepackage{lmodern}
\usepackage[T1]{fontenc}
\usepackage{textcomp}

\usepackage[utf8]{inputenc}
\usepackage[unicode]{hyperref}
\usepackage{amsfonts}
\usepackage{indentfirst}
\usepackage{enumerate}
\usepackage{amssymb}
\usepackage{amsmath}
\usepackage{mathrsfs}
\usepackage{paralist}
\usepackage{amsmath,amscd}
\usepackage{amsthm}

\newtheorem{lemma}{Lemma}
\newtheorem{theorem}[lemma]{Theorem}
\newtheorem{df}[lemma]{Definition}

\newtheorem{cl}[lemma]{Claim}
\newtheorem{prop}[lemma]{Proposition}
\newtheorem{question}[lemma]{Question}
\newtheorem{rem}[lemma]{Remark}

\begin{document}

\title{Small-large subgroups of locally compact Abelian Polish groups}
\author{Márk Poór}
\address{E\"otv\"os Lor\'and University, Institute of Mathematics, P\'azm\'any P\'eter
	s. 1/c, 1117 Budapest, Hungary}
\email{sokmark@caesar.elte.hu}
\thanks{The author was supported by the National
	Research, Development and Innovation Office – NKFIH, grants no. 104178, 124749 and 129211. \\
\raisebox{0.5cm}{\parbox[c]{11cm}{
		\scshape Supported by the \'UNKP-18-3 New National Excellence Program of the Ministry of Human Capacities.
}}}

\maketitle

\begin{abstract}
	In \cite{small-large} Ros\l{}anowski and Shelah asked whether every locally compact non-discrete group has a null but non-meager subgroup. In this paper
	we give an affirmative answer in the Polish Abelian case.

\end{abstract}

\section{Introduction}

We are interested in null, but non-meager subgroups of locally compact Abelian (LCA) Polish groups.
Only the non-discrete case is of interest, since non-empty open sets are of positive measure, thus in a discrete group $\{e \}$ has positive measure, and obviously is of second category.
It is known that under the Continuum Hypothesis one can construct null but non-meager and meager but non-null subgroups in the Cantor group or the reals, and Talagrand proved that in the Cantor group there is a null non-meager filter, that is a subgroup too \cite{talag}. Due to Fremlin and Shelah it is known that an unpublished result of H.M. Friedman implies that in the Cohen model there is no meager but non-null subgroup of the Cantor group \cite{Burke}. In \cite{small-large}, Ros\l{}anowski, and Shelah constructed $ZFC$ examples for null, but non-meager subgroups
of the reals and the Cantor group, and showed that it is consistent with $ZFC$ that  in the two groups every meager
subgroup is null. Then they asked two questions:

\textbf{Problem 5.1 \cite{small-large}}
\begin{enumerate}[(1)]
	\item Does every locally compact group (with complete Haar
	measure) admit a null non–meager subgroup?
	\item Is it consistent that no locally compact group has a meager non–null subgroup?
\end{enumerate}
In the first part of the paper we construct null but non-meager subgroups in  every non-discrete  locally compact Abelian  Polish group.  In the second section we will show that it is consistent that in each locally compact Polish group meager subgroups are always null.

\section{Preliminaries and notations}

\normalsize

We say that the topological groups $G$ and $H$ are isomorphic, in symbols $G \simeq H$, if there is an algebraic isomorphism, which is also a homeomorphism.
Under the symbol $\leq$ we mean the subgroup relation, i.e. $H \leq G$ symbols that $H$ is a subgroup of $G$,
and $H \vartriangleleft G$ is for normal subgroups.
 All topological groups are assumed to be Hausdorff.

 A topological group  $G$ is said to be locally compact if for each $g \in G$, there is a neighborhood $B$ of $g$, which is compact (i.e. each point has an open neighborhood which has compact closure).
 It is known that for a locally compact Hausdorff topological group $G$, if $C \vartriangleleft G$ is a closed normal subgroup, then
 $G/C$ is locally compact and Hausdorff. Also, if $G$ is Polish then so is $G/C$.
Since in compact Hausdorff spaces the Baire category theorem is true (i.e. no nonempty open set can be covered by countably many nowhere dense sets), it is also true in locally compact Hausdorff spaces.

For any topological space $X$, $\mathcal{B}(X)$ denotes the Borel sets, i.e. the $\sigma$-algebra generated by the open sets.

Recall the following definition of left Haar measure:
\begin{df} Let $G$ be a locally compact group, and $\mu$ be a Borel measure, i.e.
	\[ \mu : \mathcal{B}(G) \to [0, \infty] \]
	such that
	\begin{itemize}
		\item	$\mu$ is left-invariant, i.e. for every $g \in G$, $B \in \mathcal{B}(G)$
			\[ \mu(gB) = \mu(B), \]
		\item $\mu(U) > 0$ for each open $U \neq \emptyset$,
		\item $\mu(K) < \infty$ for each compact set $K$,
		\item $\mu$ is inner regular with respect to the compact sets, that is for every  Borel $B$
		  \[  \mu(B) = \sup \{ \mu(K): \ K \subseteq B, \ K \text{ is compact} \} .\]
	\end{itemize}
	Then $\mu$ is called a left Haar measure of $G$.
\end{df}
It is known that a left Haar measure always exists, and it is unique up to a positive multiplicative constant.
By slight abuse of notation we will identify $\mu$ with its completion, i.e. every set $H \subseteq G$ is measurable if $H$ differs from a Borel by (at most) a null Borel set, i.e.
\[ \text{ if } (\exists B, B' \in \mathcal{B}:  B \bigtriangleup H \subseteq B' \ \wedge \mu(B') = 0) \ \text{ then } \mu(H):= \mu(B). \]

\begin{rem}
	There is another definition of (left) Haar measure in locally compact groups, but in locally compact Polish groups (more generally in locally compact groups which are $\sigma$-compact) these two coincide.
\end{rem}

From now on $\mathcal{N}$ will denote the null-ideal w.r.t. a left Haar measure $\mu$, i.e.
\[ \mathcal{N} = \{ H \subseteq G:  \mu(H) = 0 \}. \]

It is known that for any locally compact group the null-ideals of a left Haar measure and a right Haar measure coincide, thus we can 
speak about null sets (\cite[442F ]{Fremlin}).

\vspace{1cm}

\section{Null but non-meager subgroups}

First, using abstract Fourier analytic methods we will show that the general (Polish LCA) case can be reduced to the case of the circle group, the group of $p$-adic integers, and the direct product of countably many finite (Abelian) groups. Provided that these groups have null but non-meager subgroups, we only have to show that an arbitrary LCA group has a quotient isomorphic to one of the aforementioned three groups, and then verify that pulling-back an appropriate  subgroup of the quotient will work.
The idea of first finding a nice quotient group, and constructing the desired object for that group is from \cite{elektoth}.

Since the direct product of $\omega$-many finite groups, and $\mathbb{R}$ has a null but non-meager subgroup, and so does the circle group by this (see \cite{small-large}), we only have to 
construct such subgroups in the group of $p$-adic integers.

The following definition of the $p$-adic integers is from \cite[(10.2)]{abstr}.

In Lemma $\ref{folemma}$  we prove that in any non-discrete LCA group an appropriate open subgroup has a nice quotient group.
After that, by Proposition $\ref{padikus}$ the case of $p$-adic integers is handled.

Finally we summarize, and state our main result in Theorem $\ref{fot}$.

\begin{df} \label{padikusdef}
	Let $p \in \mathbb{N}$ be a prime. Then consider the product set  
	\[ \Delta_p = \prod_{i = 0}^\infty \{ 0,1, \dots,p-1 \} \]
	with the following group operation. Let $x, y \in \Delta_p$. Define the series $t \in \prod_{i=0}^\infty \{0,1\}$, $z \in \prod_{i=0}^\infty \{ 0,1, \dots, p-1\} $  inductively so that
	\begin{equation}\label{elsojegy}
	x_0 + y_0 = t_0p + z_0 ,
	\end{equation}
	and
	\begin{equation} \label{tobbijegy}
	t_{i} + x_{i+1} +y_{i+1} = t_{i+1} p + z_{i+1}
	\end{equation}
	holds for each $i \in \omega$.
	Then $z \in \Delta_p$ is the sum of $x$ and $y$.
	
	The topology is defined by the following neighborhood base of  $0$  (i.e. the identically $0$ sequence)
	\[ \Lambda_i = \{ x \in \Delta_p \ | \ x_j = 0 \text{ for each } j<i \}  \ \ (i \in \omega ) .\]
	 Hence the sets of the form
	\begin{equation} \label{top}  y \Lambda_n = \{ x \in \Delta_p \ | \ x_j = y_j \ (j<n)  \} \end{equation}
	is a base consisting of clopen sets. In other words it is the product topology.

\end{df}


\begin{lemma} \label{folemma}
	Let $G$ be a locally compact Abelian group, which is non-discrete. Then there is an open subgroup $G' \leq G$, and a closed subgroup $C \leq G'$ of $G'$, such that the quotient group $G'/C$ is 
	isomorphic to one of the following:
	\begin{itemize}
		\item the circle group $\mathbb{T} \simeq \mathbb{R} / \mathbb{Z}$,
		\item the product of $\omega$-many finite (Abelian) groups $\prod_{i \in \omega} G_i$ (with each $G_i$ having at least two elements),
		\item the group of $p$-adic integers ($\Delta_p$) for some prime $p$.
	\end{itemize}
\end{lemma}
\begin{proof}
	By the Principal Structure Theorem of LCA groups \cite[2.4.1]{Ru}, there is an open subgroup $H$ of $G$ of
	the form $H \simeq K \times \mathbb{R}^n$ for some $n \geq 0$, where $K$ is compact. If $n>0$, then 
	$K \times (\mathbb{Z} \times \mathbb{R}^{n-1}) \vartriangleleft H$ is a closed normal subgroup of $H$, 
	\[ H / (K \times \mathbb{Z} \times \mathbb{R}^{n-1} ) \simeq (K \times \mathbb{R} \times \mathbb{R}^{n-1})/ (K \times \mathbb{Z} \times \mathbb{R}^{n-1} ) \simeq \mathbb{R}/\mathbb{Z} \simeq \mathbb{T}, \]
	we are done. Hence we can assume that $n=0 $, i.e. $H = K$ is compact. First observe that since $H \leq G$ is open, and $G$ was non-discrete, $H$ cannot be discrete. 
	
	Next we claim that the dual group $\hat{H}$ of $H$ is infinite.
	Assume on the contrary that $\hat{H}$ is finite. Then it is compact, thus the dual $\hat{\hat{H}}$ of $\hat{H}$ would be discrete (by \cite[1.2.5]{Ru}). But by the Pontryagin Duality Theorem (\cite[1.7.2]{Ru}) $\hat{\hat{H}}$ is isomorphic to $H$ which was not discrete, a contradiction.
	
	Now, by \cite[1.7.2]{Ru}, and \cite[2.1.2]{Ru} groups of the form $H/C$ (where $C$ is a closed subgroup) are the same (isomorphic as topological groups) as dual groups of the  closed subgroups of $\hat{H}$.
	This means that having a closed subgroup $L \leq \hat{H}$ guarantees that there is a closed subgroup $C \leq H$ such that the quotient group 
	\[ I = H/C \simeq \hat{L} \]
	(Note that each subgroup of $\hat{H}$ is automatically closed, since according to \cite[1.2.5]{Ru} the compactness of $H$ implies that $\hat{H}$ is discrete.)
	Now using \cite[Thm. 2.16]{elektoth}, and the infiniteness of $\hat{H}$, we obtain that $\hat{H}$ has a subgroup (let's say $L$) isomorphic to either 
	\begin{enumerate}[(i)]
		\item $\mathbb{Z}$, or
		\item $\oplus_{i \in \omega} G_i$ where each $G_i$ is a finite group of at least two elements, or
		\item the quasicyclic group $C_{p^\infty}$ for some prime $p$.
	\end{enumerate}

	This means that 
	\begin{enumerate}[(i)]
		\item if $L \simeq \mathbb{Z}$, then 
		 there is a closed subgroup $C \leq H$ such that
		\[ H/C \simeq \hat{L} \simeq \mathbb{T}. \]
		
		\item If $ L \simeq \oplus_{i \in \omega}G_i$ (where each $G_i$ has at least two elements):  
		First note that an LCA group is compact and discrete at the same time iff so is its dual by \cite[1.2.5]{Ru} and \cite[1.7.2]{Ru}. This means that an LCA group is finite iff its dual is finite.
		Second, since the dual of the dual of an LCA group is the group itself \cite[1.7.2]{Ru}, if $\hat{G_i} = \{ 0 \} $		was the trivial group, then $G_i$ also would  be the trivial group, therefore each $\hat{G_i}$ has at least two elements.
		 From this we obtain that each $\hat{G_i}$ is finite, and is of cardinality at least two, and using that the dual of the direct sum (endowed with the discrete topology) is the product of the dual groups \cite[2.2.3]{Ru} we get that 
		 \begin{equation} \label{dir}
		  (\oplus_{i \in \omega}G_i)\hat{} \simeq \prod_{i \in \omega} \hat{G_i} .
		  \end{equation}
		  Thus there is a closed subgroup $C \leq H$ such that 
		  \[ H/C \simeq  \hat{L} \simeq \prod_{i \in \omega} \hat{G_i} \ \ \  ( \forall i: \ |\hat{G_i}| >1 ). \]
		  
		\item If $L \leq \hat{H}$ is isomorphic
		to the quasicyclic group $C_{p^\infty}$ (for some prime $p$), then as $\hat{\Delta_p} \simeq C_{p^\infty}$  \cite[25.2]{abstr} we obtain that 
		\[ \hat{L} \simeq \hat{C}_{p^\infty} \simeq \Delta_p \]
		by \cite[1.7.2]{Ru}.
	    Therefore there is a closed subgroup $C \leq H$ for which
	     \[ H/C \simeq \hat{L} \simeq \Delta_p. \]

	\end{enumerate}

	This shows that there exists a closed subgroup $C \leq H$ such that $I \simeq H/C$ is either the circle group, 
	the $p$-adic integers, or the product of $\omega$-many finite groups, each finite group having at least two elements.
	
\end{proof}

\begin{prop} \label{padikus}
	Let $p \in \mathbb{N}$ be a prime. Then the group of $p$-adic integers has a null but non-meager subgroup.
\end{prop}
\begin{proof}
Fix a non-principal ultrafilter $\mathcal{U}$ on $\omega$, and for each $j \in \omega$, $0 \leq k < (j+1)^2- j^2$ define the intervals 
\[ I^k_{j} = [j^2+k, (j+1)^2) \subseteq \omega. \]
Let $\mu$ denote the probability Haar measure  (i.e. $\mu(\Delta_p) =1$), and 
\begin{equation} \label{Hdf} H_k =  \left\{ x \in \Delta_p \ | \ \{ j \in \omega \ | \ k< (j+1)^2 - j ^2 \ \wedge \ (x_{|I^k_j} \equiv 0 \vee x_{|I^k_j} \equiv p-1)  \} \in \mathcal{U} \right\} 
 \end{equation}
i.e. those which are constant $0$-s or $p-1$-s on $\mathcal{U}$-almost every $I^k_j$-s.
(Note that since  $I^k_j \supseteq I^{k+1}_j$ if $j$ is such that both $I^k_j$ and $I^{k+1}_j$ are defined, 
\begin{equation} \label{mon}
H_k \subseteq H_{k+1} \text{ .)}
\end{equation}
Define the set $H$ as follows
\[ H = \bigcup_{k \in \omega} H_k = \]
\[ =  \left\{  x \in \Delta_p \ | \ \exists k: \ \{ j \in \omega \ | \ (j > (k-1)/2) \ \wedge \  (x_{|I^k_j} \equiv 0 \vee x_{|I^k_j} \equiv p-1) \} \in \mathcal{U} \right\}. \]
We claim that $H$ is the desired subgroup.
\begin{enumerate}[(i)] 
	\item $H$ is a subgroup.
	
	Let $x,y \in H$. Using $\eqref{mon}$ let $n \in \omega$ be such that $x,y \in H_n$.
	Consider the following elements of $\mathcal{U}$
	\[ U_x = \{ j \in \omega\ | \   x_{|I^n_j} \equiv 0 \vee x_{|I^n_j} \equiv p-1 \} \in \mathcal{U}, \]
	\[U_y = \{ j \in \omega\ | \   y_{|I^n_j} \equiv 0 \vee y_{|I^n_j} \equiv p-1 \} \in \mathcal{U}, \]
	and let $U = U_x \cap U_y \in \mathcal{U}$ denote their intersection.
	Now, if $j \in U$, then the following lemma will ensure that $(x+y)_{|I^{n+1}_j } \equiv 0$, or $(x+y)_{|I^{n+1}_j} \equiv p-1$, which
	means that $x+y \in H^{n+1} \subseteq H$.
	\begin{lemma}
		Let $x,y \in \Delta_p$ and $m<l \in \omega$ such that
		\[ x_{|[m,l)]} \equiv 0 \text{, or } x_{|[m,l)]} \equiv p-1 \]
		and
		\[ y_{|[m,l)]} \equiv 0 \text{, or } y_{|[m,l)]} \equiv p-1 .\]
		Then
		\[ (x+y)_{|[m+1,l)]} \equiv 0 \text{, or } (x+y)_{|[m+1,l)]} \equiv p-1 .\]
	\end{lemma}
	\begin{proof}
		Let $t \in \prod_{i=0}^\infty \{0,1\}$, $z= x+y \in \prod_{i=0}^\infty \{ 0,1, \dots, p-1\} $ as in Definition $\ref{padikusdef}$, i.e. $\eqref{elsojegy}$ and $\eqref{tobbijegy}$ hold.
		Now we have three cases.
		\begin{itemize}
			\item Case 1. If $x_{|[m,l)]} =   y_{|[m,l)]} \equiv 0$,
			
			then $x_m+y_m+t_{m-1} =t_{m-1} \in \{0,1\}$, thus the recursive definition of $t$ and $z$ $\eqref{tobbijegy}$ implies that $t_m =0$. From this it is easy to see that $z_{m+1}=z_{m+2} = \dots = z_{l-1} =0$, i.e. $z_{|[m+1,l)]} \equiv 0$, as desired.
			\item Case 2. If $x_{|[m,l)]} \equiv 0 $, and    $y_{|[m,l)]} \equiv p-1$,
			
			then $z_{|[m+1,l)}$ depends only on $t_{m-1}$. If $t_{m-1} = 0$, then 
			\[ x_m+y_m+t_{m-1} = p-1 = 0 \cdot p + p-1 ,\]
			hence $ z_m = p-1 $,
			and $t_m = 0$. Using $\eqref{tobbijegy}$ it is straightforward to check that 
			\[ t_{m+1} = t_{m+2} = \dots = t_{l-1} = 0 ,\]
			and 
			\[ z_{m} = z_{m+1} = \dots = z_{l-1} = p-1. \]
			On the other hand, if $t_{m-1} = 1$, then 
			\[ x_m+y_m+t_{m-1} = p-1 +1 = 1 \cdot p + 0 ,\]
			i.e. $t_m= 1$, and $z_m = 0$. Similarly, applying $\eqref{tobbijegy}$ $l-m-1$ times, one can get that 
			\[ t_m=  t_{m+1} = t_{m+2} = \dots = t_{l-1} = 1 ,\]
			and 
			\[ z_{m} = z_{m+1} = \dots = z_{l-1} = 0. \]
						
			\item Case 3. If $x_{|[m,l)]} =   y_{|[m,l)]} \equiv p-1$,
			
			then 
			\[ x_m+y_m+t_{m-1} =(p-1) + (p-1) + t_{m-1}  = 1 \cdot p + p-2 + t_{m-1}, \]
			i.e. $t_m = 1$. Therefore we got that
			\[ x_{m+1}+y_{m+1}+t_{m} =(p-1) + (p-1) + 1  = 1 \cdot p + p-1, \]
			i.e. $z_{m+1} =p-1$, $t_{m+1} = 1$. Iterating this argument yields that
		  	\[ t_{m+1} = t_{m+2} = \dots = t_{l-1} = 1 ,\]
		  	and 
		  	\[ z_{m+1} = z_{m+2} = \dots = z_{l-1} = p-1. \]
			
		\end{itemize}
	\end{proof}

	It is left to show that $H$ is closed under taking inverse, i.e. if $x \in H_n$, then $-x \in H$.
	We will show that if $i>0$ and then $(-x)_i = p-1 - x_i$  from which we can conclude that
		\[ U_x = \{ l \in \omega\ | \ l > (k-1)/2 \wedge  (x_{|I^n_l} \equiv 0  \ \vee \  x_{|I^n_l} \equiv p-1) \} \in \mathcal{U} ,\] 
	then (since $0 \notin I_l^j$ if $l >0$)
	\[  \{ l \in \omega \setminus \{0\}\ | \ l > (k-1)/2 \wedge ((-x)_{|I^n_l} \equiv 0  \ \vee \  (-x)_{|I^n_l} \equiv p-1) \} \supseteq \]
	\[ \supseteq  U_x \setminus \{0 \} \in \mathcal{U}. \] 
	
	Now we have to show that if we define $w \in \Delta_p$ so that $w_0 = p - x_0$, $w_{i} = p-1-x_i$) ($i>0$), then $w + x = 0$ in $\Delta_p$.
	 Let $t \equiv 1 \in \prod_{i=0}^{\infty} \{0,1\}$ the constant $1$ function.
	 Now it is easy to see that $x_0 + w_0 = t_0 p$ and $x_i+w_i+1 = t_ip$, thus according to the rule of addition in $\Delta_p$, ($\eqref{elsojegy}$, $\eqref{tobbijegy}$), $x + w = 0$, i.e. $w = -x$.
	 \\

	\item $H$ is null.
	
	Since
	\[  H = \bigcup_{k \in \omega} H_k,  \]
	it is enough to show that $\mu(H_k)=0$ for each $k$.
	But using the non-principality of $\mathcal{U}$
	\[ \mu(H_k) = \]
	\[ = \mu\left( \left\{ x \in \Delta_p \ | \ \{ j \in \omega \ | \ j>(k-1)/2 \ \wedge \ (x_{|I^k_j} \equiv 0 \vee x_{|I^k_j} \equiv p-1)  \} \in \mathcal{U} \right\} \right) \leq \]
	\[ \leq \mu \left( \left\{ x \in \Delta_p \ | \ \left| \{ j \in \omega \ | \ j>(k-1)/2 \ \wedge \ (x_{|I^k_j} \equiv 0 \vee x_{|I^k_j} \equiv p-1)  \} \right| = \infty  \right\} \right) \]

	i.e. it suffices to show that those sequences, for which there exist infinitely many $j$-s such that the sequence is constant $0$ or $p-1$ on $I^k_j$ form a null set.
	Now, we know that for each $m$ the $p^m$-many translates of the open set 
	\[ \Lambda_m = \{ x \in \Delta_p \ | \ x_l=0 \ \text{ for each } l < m \} \]
	is a partition of $\Delta_p$ $\eqref{top}$, therefore because $\mu$ is a translation-invariant probability measure
	\[ \mu(y \Lambda_m) = \mu( \{x \in \Delta_p| \ x_l = y_l \ (\forall l < m) \} )= \frac{1}{p^m}, \]
	hence
	\[ \mu( \{ x \in \Delta_p \ | \ x_{I^k_j} \equiv 0 \} ) = \mu( \{ x \in \Delta_p \ | \ x_{I^k_j} \equiv p-1 \} ) = \frac{1}{p^{|I^k_j|}} = \frac{1}{p^{2j+1-k}} .\]
	Since this gives us that
	\[\mu( \{ x \in \Delta_p \ | \ x_{I^k_j} \equiv 0 \vee x_{I^k_j} \equiv p-1  \} ) = \frac{2}{p^{2j+1-k}} , \]
	for arbitrary fixed $l_0$ the following inequality will hold
	\[  \mu \left( \left\{ x \in \Delta_p \ | \ \left| \{ j \in \omega \ | \ j> (k-1)/2 \ \wedge \ (x_{|I^k_j} \equiv 0 \vee x_{|I^k_j} \equiv p-1)  \} \right| = \infty  \right\} \right)  = \]
	\[ = \mu \left( \bigcap_{l=0}^{\infty} \bigcup_{j=l}^{\infty}  \left\{ x \in \Delta_p \ |  (x_{|I^k_j} \equiv 0 \vee x_{|I^k_j} \equiv p-1)   \right\} \right) \leq \]
	\[ \leq  \mu \left( \bigcup_{j=l_0}^{\infty} \left\{ x \in \Delta_p \ |  (x_{|I^k_j} \equiv 0 \vee x_{|I^k_j} \equiv p-1)   \right\} \right)  .\]
	 We obtained the following obvious upper bound
	\[  \mu \left( \bigcup_{j=l_0}^{\infty}  \left\{ x \in \Delta_p \ |  (x_{|I^k_j} \equiv 0 \vee x_{|I^k_j} \equiv p-1)   \right\} \right) \leq \sum_{j=l_0}^{\infty} \frac{2}{p^{2j+1-k}} ,\]
	and as the latter tends to $0$ when $l_0$ tends to infinite, we are done.
	\item $H$ is non-meager. 
	
	We will show that even $H_0$ is non-meager.
	For each $j \in \omega$ define the mapping $f_j$ as follows
	\[ f_j :  \{0,1, \dots p-1 \}^{|I^0_j|} \to \{0,1   \} \]
	 \begin{equation} \label{fjdef} \underline{v} \mapsto \left\{ \begin{array}{cc} 0 & \text{ iff } v \equiv 0 \\ 1 &  \text{ otherwise} \end{array} \right.  \end{equation}
	 Let $f$ be the following mapping from $\Delta_p$ to $2^\omega$ (where we identify $\Delta_p$ with $\{ 0,1, \dots p-1\}^\omega$)
	 \[ f: \Delta_p \to 2^\omega \]
	 \begin{equation}  \label{fdef} x \mapsto (f_j(x_{|I^0_j} ))_{j_ \in \omega} \end{equation}
	Now suppose that $\Delta_p \setminus H$ is co-meager. As $f$ is a continuous surjective open mapping between Polish spaces, according to 	\cite[Lemma 2.6 ]{levelsets} $f$ maps a co-meager set onto a co-meager set, let
	$R = f(\Delta_p \setminus H)$ denote this co-meager set in $2^\omega$.
	Now, using \cite[Thm. 2.2.4 ]{Bart}), there exist a strictly growing sequence of non-negative integers 
	\[ 0 = n_0  <n_1 < \dots, \]
	and an element $r \in 2^\omega$ such that
	\begin{equation} \label{vegt} \left\{ s \in 2^\omega \ | \ |\{ j : \ s_{|[n_j,n_{j+1}) } \equiv r_{|[n_j,n_{j+1}) } \}| =  \infty \right\} \subseteq R  . \end{equation}
	Now the following disjoint sets cover $\omega$ (since the sequence $n_j$ is strictly growing)
	\begin{equation} \label{Ui} U_0 = \bigcup_{k \in \omega} [n_{2k}, n_{2k+1}), \ \ U_1 = \bigcup_{k \in \omega} [n_{2k+1}, n_{2k+2}), \end{equation}
	thus exactly one of them is in $\mathcal{U}$.
	Let $U$ denote that set. 
	Now define the following element $y$ of $2^\omega$
	\begin{equation} \label{ydefje} y_i = \left\{ \begin{array}{cc} 0 & \text{ iff } i \in U \\ r_i &  \text{ otherwise} \end{array} \right.\end{equation}
	Clearly (using $\eqref{Ui}$ and the fact that $U = U_i$ for some $i \in \{0,1\}$) there are infinitely many $j$-s for which $y_{|[n_j,n_{j+1})} \equiv r_{|[n_j,n_{j+1})}$, thus
	by $\eqref{vegt}$ 
	\[ y \in R = f(\Delta_p \setminus H) \]
	Hence, there is a \begin{equation} \label{nemeleme} z \in \Delta_p \setminus H, \end{equation} for which
	\begin{equation} \label{zd} f(z) = y. \end{equation}
	But for each $j \in \omega$, if
	\[j \in U \]
	\[ \Downarrow  \text{ (by }\eqref{fdef} \text{ and } \eqref{ydefje} \text{) } \]
	\[ y_j = f_j(z_{|I^0_j}) =  0 \]
	\[ \Downarrow \text{ (by }  \eqref{fjdef} \text{)}   \]
	\[  z_{|[j^2,(j+1)^2) } \equiv 0 .\]
\end{enumerate}
This means, as $U \in \mathcal{U}$, that  we found that for $\mathcal{U}$-almost every $j$,
$z_{|[j^2,(j+1)^2) } \equiv 0$, therefore $z \in H_0$  by the definition of $H_0$ $\eqref{Hdf}$, contradicting $\eqref{nemeleme}$.
\end{proof}

Now we are ready to state our result.

\begin{theorem} \label{fot}
	Let $G$ be a non-discrete Polish LCA group. Then there is a null, non-meager subgroup in $G$.
\end{theorem}
\begin{proof}
	According to Lemma $\ref{folemma}$, there is an open subgroup $G' \leq G$ for which there exists a closed normal subgroup $C \vartriangleleft G'$ such that $G'/C$ is isomorphic to one of the following
	\begin{enumerate}[(i)]
		\item \label{a} the circle group $\mathbb{T} \simeq \mathbb{R} / \mathbb{Z}$,
		\item \label{b} the product of $\omega$-many finite (Abelian) groups $\prod_{i \in \omega} G_i$ (with each $G_i$ having at least two elements),
		\item the group of $p$-adic integers ($\Delta_p$).
	\end{enumerate}

First we show that 
\begin{cl}
	$G'/C$ has a null but non-meager subgroup, say $H/C$ (where $C \leq H \leq G'$).
\end{cl}
\begin{proof}

In  case $\eqref{a}$, first, consider the
null non-meager subgroup $L \leq \mathbb{R}$  given by \cite[Theorem 2.3]{small-large}. Recall that the
Haar measure is the Lebesgue measure on the real line (up to a positive multiplicative constant). After identifying the circle group with $[0,1)$, the Haar measure on it coincides with the Lebesgue measure restricted to 
the unit interval. Therefore the set
\[ L' = \{ r \in [0,1) : \ \exists k \in \mathbb{Z} \ \  k+r \in L \}  \subseteq [0,1)= \mathbb{T}  \]
is a null set, and forms a subgroup in the circle group. $L' \subseteq [0,1)$
 is non-meager, since the meagerness of $L'$ would imply that $L' + \mathbb{Z} \supseteq L$ was meager in $\mathbb{R}$. 
 Given this null and non-meager subgroup $L' \leq \mathbb{T} \simeq G'/C$, let $H \leq G'$ be such that $H/C  \leq G'/C \simeq \mathbb{T}$ is a null, but nonmeager subgroup.
 
In  case  $\eqref{b}$, i.e. 
\[ G'/C \simeq \prod_{i \in \omega} F_i  \ \ (\forall i: \ 1<|F_i| < \infty) \]
\cite[Remark 4.2]{small-large} implies that there is a null but non-meager subgroup
in $G'/C$, say $H/C$.

If $G'/C$ is isomorphic to the group of the $p$-adic integers for some prime $p$, then Proposition $\ref{padikus}$ guarantees that there is an appropriate subgroup in $G'/C$.
\end{proof}
Next we have to verify that from the nullness and non-meagerness of $H/C$ (in $G'/C$) follows the
nullness and non-meagerness of $H$ in $G'$.
Let $\pi: G' \to G'/C$ denote the canonical projection.
 Now since a Haar measure $\nu$ on $G'/C$ is also $G'$-invariant and Radon, using \cite[443P/(c)]{Fremlin} (with $X= G'$, $Y= C$, $\lambda= \nu$)
we obtain that $H \leq G'$ is a subgroup of measure zero.

For the non-meagerness, suppose that $G' \setminus H$ is co-meager in $G'$. Then using  \cite[Lemma 2.6]{levelsets} (and the fact $C \leq H$), the image  of $G' \setminus H$ under the canonical projection $\pi: G' \to G'/C$ 
\[ \pi(G' \setminus H)  = (G' /C) \setminus (H/C) \]
would also be co-meager (in $G'/C$), contradicting that $H/C$ is non-meager.

Now, we have a null but non-meager subgroup $H$ in $G'$, but since $G' \leq G$ is open, it is straightforward to check that a Haar measure of $G$ restricted to $G'$ is a Haar measure of $G'$, therefore $H$ is null in $G$ too. Moreover a non-meager set 
in an open subspace (namely $G'$) is also non-meager in the whole space, thus $H \leq G$ is a null non-meager subgroup.

\end{proof}

\section{Meager but non-null subgroups}

The following result can be found in \cite{en}, but for the sake of completeness we include the sketch of the proofs. We will prove that it is consistent with $ZFC$ that in every locally compact Polish group meager subgroups are null.

We will use H. M. Friedman's theorem, which can be found in \cite{Burke}:
\begin{theorem}[H. M. Friedman] \label{Friedm}
	In the Cohen model (that is, adding $\omega_2$ Cohen reals to a model of $ZFC + CH$)
	the following holds:
	\[ \begin{array}{l} 
	 \forall  H \subseteq  2^\omega \times 2^\omega \ F_\sigma: \\
	\text{ if } (\exists C \times D \subseteq H, \ C \times D \notin \mathcal{N}), \text{ then } (\exists A \times B \subseteq H, \ A \times B \notin \mathcal{N}  \text{ measurable}) \\
	
	\end{array} \]
\end{theorem}

This property of $2^\omega$ implies that every meager subgroup of $2^\omega$ is of measure zero. Similarly if it holds in a locally compact Polish group $G$, then $G$ has no meager non-null subgroups, this lemma is also from \cite{Burke} (stating it only for $G= 2^\omega$, but the proof is the same).

\begin{lemma} \label{Fszigm}
	Let $G$ be a locally compact  Polish group, $\mu$ be a left Haar measure on $G$. Assume that the following holds:
	\begin{equation} \label{tegla}
	\begin{array}{l} 
	\forall H \subseteq G \times G, \ F_\sigma: \\
	(\exists C \times D \subseteq H, \ C \times D \notin \mathcal{N}_\mu) \longrightarrow (\exists A \times B \subseteq H, \ A \times B \notin \mathcal{N}_\mu \text{ measurable}) \\
	\end{array}
	\end{equation} 
	Then every meager subgroup of $G$ is null.
\end{lemma}
\begin{rem} \label{merh}
	If $X$ is a Polish space, $\nu$ is a $\sigma$-finite Borel measure on $X$, then for every $\nu$-measurable set $H$ there exists a Borel $B$ such that
	$H \bigtriangleup B$ is null (wrt. $\nu$).
\end{rem}

In Lemma $\ref{C-lengyel}$ we show that if the condition of Lemma $\ref{Fszigm}$ holds in $2^\omega$ (e.g. in the Cohen model), then this holds for arbitrary locally compact Polish groups. This yields that it is consistent with $ZFC$ that in locally compact Polish groups  meager subgroups are always null.



\begin{lemma} \label{C-lengyel}
	Assume that the condition from Lemma $\ref{Fszigm}$  holds in $2^\omega$ (i.e. every $F_\sigma$ subset of $2^\omega \times 2^\omega$ which contains a non-null rectangle must contain a measurable non-null rectangle).  Then this condition holds in every locally compact Polish group $G$.
\end{lemma}
Combining Theorem $\ref{Friedm}$ with Lemma $\ref{Fszigm}$ yields the following.
\begin{theorem} \label{fov}
	It is consistent with $ZFC$ that in every locally compact Polish group meager subgroups are always null.
\end{theorem}
\begin{rem} \label{tegl}
	If $X$ is a Polish space, and $\mu$ is a $\sigma$-finite Borel measure, then for any rectangle $C \times D \subseteq X \times X$,
	it is non-null (wrt. the product measure $\mu \times \mu$) iff $C \notin \mathcal{N}_\mu$ and $D \notin \mathcal{N}_\mu$.
	
\end{rem}

\begin{proof}(Lemma $\ref{C-lengyel}$)
	Assume that $\mu$ denotes the left Haar measure on $G$, $\nu$ denotes the Haar measure on $2^\omega$. Then
	since locally compact Polish groups are $\sigma$-compact, the Haar measure $\mu$ is $\sigma$-finite. 
	\begin{lemma}
	There exist $F_\sigma$ sets $C \subseteq 2^\omega$ and $K \subseteq G$ of full measure, and a Borel bijection $f: C \to K$ such that for each $S \subseteq C$
	\begin{enumerate}[(i)]
			\item \label{nullt} $\nu(S)=0$ iff $\mu(f(S))=0$,
			\item \label{Fszt} $S$ is measurable iff $f(S)$ is measurable,
	\end{enumerate}
	and for the bijection $f \times f: C \times C \to K \times K$
	\begin{equation} \label{fkf}
	  S' \subseteq C \times C \text{ is }F_\sigma \quad \text{ iff } \quad (f \times f)(S') \subseteq K \times K \text{ is } F_\sigma.
	\end{equation}
	\end{lemma}
	
	\begin{proof}
		
	For the construction of $f$ we will use the following claim. One can prove it using the regularity of $\nu$ and $\mu$, the isomorphism theorem for measures \cite[Thm 17.41.]{Kechris} and  Lusin's theorem (see \cite[Thm 17.12]{Kechris}).
	\begin{cl} \label{clejm}
		There is a sequence of pairwise disjoint compact sets $C_i$ ($i \in \omega$) in $2^\omega$, for which
		$\nu(C_i)>0$ and $\nu(2^\omega \setminus \bigcup_{i \in \omega} C_i) = 0$, and similarly a sequence of pairwise disjoint compact sets $K_i$ ($i \in \omega$)
		in $G$ with $\mu(K_i)>0$ and $\mu(G \setminus \bigcup_{i \in \omega} K_i)=0$, and there exist  homeomorphisms $f_i: C_i \to K_i$,
		and positive constants $r_i$ such that
		\[ \forall B \subseteq C_i \text{ Borel: } \ \nu(B) = r_i \mu(f_i(B)). \]
	\end{cl}

	Now let $C = \bigcup_{i \in \omega} C_i \subseteq 2^\omega$, $K = \bigcup_{i \in \omega} K_i \subseteq G$,  $f= \bigcup_{i \in \omega} f_i$ given by Claim $\ref{clejm}$. Then
	\[ f: C \to K \]
	is a Borel bijection, and $C$, $K$ and $f$ satisfy conditions $\eqref{nullt}-\eqref{Fszt}$ (by Remark $\ref{merh}$, $S$ differs from a Borel by a null set, $f^{-1}$ is a Borel function, and $f$ maps a null set to a null set). 
	For $\eqref{fkf}$ note that $f \times f = \bigcup_{i,j \in \omega }f_i \times f_j$ is the union of homeomorphisms between compact sets. 
	\end{proof}
		Observe that $K \times K \subseteq G \times G$, $C \times C \subseteq 2^\omega \times 2^\omega$ are $F_\sigma$ sets of full measure.
	
	Fix an $F_\sigma$ set $F \subseteq G \times G$, and let $D \times E \subseteq F$ be a non-null rectangle. Since by Remark $\ref{tegl}$
	$D$ and $E$ are non-null sets and $\mu(G \setminus K)=0$, $(D \cap K) \times (E \cap K)$ is a non-null rectangle.
	Moreover, since $K \times K$ is an $F_\sigma$ set of full measure (w.r.t. $\mu \times \mu$),
	\[ (D \cap K) \times (E \cap K) \subseteq F \cap (K \times K), \]
	i.e. the $F_\sigma$ set $F \cap (K \times K)$ contains a non-null rectangle, therefore we can assume that $F \subseteq K \times K$, it suffices to find measurable non-null rectangles in such $F$-s.
	
	If $D \times E \subseteq F$ is a non-null rectangle, then $f^{-1}(D) \times f^{-1}(E)$ is non-null by $\eqref{nullt}$.
	Now since $(f \times f)^{-1}(F) \subseteq 2^\omega \times 2^\omega$ is $F_\sigma$ (by $\eqref{fkf}$) containing the non-null rectangle $f^{-1}(D) \times f^{-1}(E)$, it also contains a product $A \times B$ with $A$, $B$ measurable, $\nu(A), \nu(B) >0$. Then  by $\eqref{nullt}$, $\eqref{Fszt}$ $f^{-1}(A)$, $f^{-1}(B)$ are measurable, $\mu$-positive subsets of $G$ such that $f^{-1}(A) \times f^{-1}(B) \subseteq F$, as desired.

\end{proof}

In \cite{en} we gave affirmative answers to both questions from \cite[Problem 5.1]{small-large} in the general (locally compact, not necessarily Polish) case. However, due to Christensen 
we can extend the notion of the null ideal to every (not necessarily locally compact) Polish group. 
We say that a subset $X$ of a Polish group is Haar-null (in the sense of Christensen), if there is a Borel probability measure $\mu$ on $G$, and a Borel set $B \supseteq X$, such that for each $g,h \in G$ $\mu(gBh) = 0$ \cite{Christ}. The following question is open.
\begin{question}

	What can we say about  small-large subgroups of non-locally compact Polish groups, replacing "null wrt. the Haar measure" by "Haar-null"? Do (always) exist Haar-null but non-meager, and meager but non-Haar-null subgroups in non-locally compact Polish groups?





\end{question}

\end{document}